\documentclass[11pt,a4paper]{amsart}
\usepackage{amssymb,amsmath,amsthm}

\theoremstyle{plain}
\newtheorem{definition}{Definition}[section]
\newtheorem{theorem}{Theorem}
\newtheorem{example}[definition]{Example}
\newtheorem{lemma}[definition]{Lemma}

\newtheorem{remark}[definition]{Remark}

\newcommand{\N}{\mathbb{N}}
\newcommand{\Q}{\mathbb{Q}}
\newcommand{\R}{\mathbb{R}}

\newcommand{\eps}{\varepsilon}
\def\step#1#2{\par\noindent{\underline{\it Step~#1.}}\emph{ #2}\\}
\renewcommand{\phi}{\varphi}
\renewcommand{\rho}{\varrho}

\title{Duality for rectified Cost Functions}
\author{Mathias Beiglb\"ock, Aldo Pratelli}

\begin{document}

\begin{abstract}
It is well-known that duality in the Monge-Kantorovich transport problem holds true  provided that the cost function $c:X\times Y\to [0,\infty]$ is lower semi-continuous or finitely valued, but it may fail otherwise. We present a suitable notion of \emph{rectificaton} $c_r$ of the cost $c$, so that the Monge-Kantorovich duality holds true replacing $c$ by $c_r$. In particular, passing from $c$ to $c_r$ only changes the value of the primal Monge-Kantorovich problem. Finally, the rectified function $c_r $ is lower semi-continuous as soon as $X$ and $Y$ are endowed with proper topologies, thus emphasizing the role of lower semi-continuity in the duality-theory of optimal transport.
\end{abstract}

\maketitle

\section{Introduction}
\subsection{Description of the main question}

We consider the \emph{Monge-Kan\-to\-ro\-vich transport problem} for
Borel probability measures  $\mu,\nu$ on Polish 
spaces $X,Y$. Standard references for the theory of optimal transportation are~\cite{Vill03,Vill09}.

The set $\Pi(\mu,\nu)$ consists of all 
\emph{transport plans}, that is,  Borel probability measures on
$X\times Y$ which have $X$-marginal $\mu$ and $Y$-marginal $\nu$.
The \emph{transport cost} associated to a \emph{cost function} $c:X\times Y\to [0,\infty]$ and a transport plan $\pi$
is given by
\begin{equation}\label{CostFunctional}
\langle c,\pi\rangle =\iint_{X\times Y}c(x,y)\,d\pi(x,y).
\end{equation}
The (primal)
Monge-Kantorovich problem is then to determine the value
\begin{equation}\label{G1}
P_c:= \inf\{ \langle c, \pi\rangle:\pi\in \Pi(\mu,\nu)\}\, .
\end{equation}
and to identify a primal optimizer $\hat{\pi} \in \Pi(\mu,\nu)$.

A natural condition which guarantees the existence of a primal optimizer is that the cost function $c$ is lower semi-continuous. (See for instance \cite[Theorem 4.1]{Vill09}.)

To formulate the dual problem, we let
\[
\phi\oplus \psi\,(x,y):= \phi(x)+\psi(y)
\]
for functions $\phi, \psi$ on $X$ (resp.\ $Y$). 
The dual Monge-Kantorovich problem then consists in determining
\begin{equation}\label{SimpleJ}  
D_c:=\sup\left\{\int\varphi\,d\mu+\int\psi\,d\nu:  \varphi\in L^1_\mu (Y), \psi\in L^1_\nu(Y), \varphi\oplus\psi\leq c \right\}. 
\end{equation}
Given two functions $\phi,\psi$ which are integrable with respect to $\mu$ and $\nu$ respectively, and which satisfy $\phi\oplus\psi\leq c$, and given a transport plan $\pi\in \Pi(\mu,\nu)$ we clearly have
\[
\iint c\,d\pi\geq \iint \phi\oplus \psi\, d\pi= \int \phi \,d\mu+\int \psi\, d\nu\,,
\]
hence it follows that $P_c\geq D_c$. The question if there actually is \emph{equality}, i.e.\ whether Monge-Kantorovich duality  $P_c=D_c$ holds true, has been intensively studied in the years by many authors, see for instance~\cite{Kant42,KaRu58,Dudl76,Dudl02,deAc82,GaRu81,Fern81,Szul82,Mika06,MiTh06}, and see also the bibliographical notes in \cite[p86, 87]{Vill09}. In particular, it is known that $P_c=D_c$ provided that the cost function $c$ is \emph{lower semi-continuous} (cf.\ \cite[Theorem 2.6]{Kell84} or \cite[Theorems 5.10]{Vill09} for a modern source), or merely measurable but \emph{bounded} (\cite[Corollary 2.16]{Kell84}) or at least \emph{$\mu\otimes \nu$-a.s. finitely valued} (\cite[Theorem 1]{BeSc08}). However, the duality does not hold in complete generality as simple examples show.

\begin{example}\label{ZeroOneInfty}
Let $X=Y=[0,1]$ and let $\mu=\nu$ be the Lebesgue measure. Define $c$ on $X\times Y$ to be $0$ below the diagonal, $1$ on the diagonal and $\infty$ else, i.e.,\
\begin{align*}
c(x,y)=\left\{
\begin{array}{cl}
0,&\mbox{ for }  0\le y<x\le 1,\\
1,&\mbox{ for }  0\le x=y\le 1,\\
\infty,&\mbox{ for }  0\le x<y\le 1.
\end{array}\right.
\end{align*}
The only finite transport plan is concentrated on the diagonal, hence $P_c=1$. On the other hand, if $\phi:X\to [-\infty,\infty), \psi:Y\to [-\infty, \infty)$ satisfy $\phi\oplus\psi\leq c$, one readily verifies that $\varphi(x)+\psi(x)>0$ can hold true for at most countably many $x\in[0,1]$. Hence $D_c=0$
so that there is a duality gap.
\end{example}

Let us discuss the example above a little bit. Strictly speaking, one should simply say that it presents a situation where the duality does not hold true. But on the other hand, one would like to say that in fact the duality should hold true, and it fails only because the cost function $c$ takes the ``wrong'' value on the diagonal, while the ``correct'' cost function should be
\begin{equation}\label{ExampleCR}
c_r(x,y)=\left\{
\begin{array}{cl}
0&\mbox{ for }  0\le y\leq x\le 1,\\
\infty&\mbox{ for }  0\le x<y\le 1.
\end{array}\right.
\end{equation}
In fact, in some sense, around the points in the diagonal there are ``many'' points where $c=0$, hence it makes no sense to have $c=1$ in the diagonal. Notice that with the cost function $c_r$ duality holds, and in particular
\begin{equation}\label{correctduality}
P_{c_r} = D_{c_r} = D_c\,.
\end{equation}
Basically, we are saying that in the above example the correct value of both the primal and the dual problem ``should be'' the same, namely $0$, and it is not so only because the cost function $c$ has been defined in a slightly meaningless way. In particular, the fact that $D_{c_r}=D_c$ is saying that the dual problem is less sensitive to the ``mistakes'' in the definition of $c$, while the primal problem is more sensitive and indeed $P_c > P_{c_r}$.\par\medskip
The aim of the present paper is to show that the situation is always the one described by means of the above simple example. More precisely, we will show that for any transport problem it is possible to define a meaningful \emph{rectified} cost function $c_r\leq c$, and~(\ref{correctduality}) always holds true. Roughly speaking, this means that duality in the Monge-Kantorovich problem \emph{always} holds true, as soon as one considers the ``correct'' definitions of the cost functions $c$. Moreover, an ``incorrect'' definition may only affect the value of the primal problem, and can be corrected by passing to a suitable rectified cost function $c_r$.\par\bigskip

Let us now describe another important feature of our rectification procedure. Consider a simple variant of Example~\ref{ZeroOneInfty}, where the value $+\infty$ in the definition of $c$ is replaced by some number $1<M \in \R$. In this case the cost is finite, then according to the classical results we know that the duality holds. However, the transport problem has now another drawback, namely that there are no optimal transport plans. In fact, the infimum of the costs of the transport plans is now $0$, but every transport plan has a strictly positive cost. In particular, every optimizing sequence of transport plans converges to the plan concentrated on the diagonal, which has cost $1$. Clearly, also this bad behaviour disappears if one passes to the rectified cost function $c_r$, which has value $0$ in the diagonal and coincides with $c$ outside.\par\smallskip
We will show that also this pleasant feature of the rectification process holds in general, that is, the transport problem with the rectified cost $c_r$ always admits optimal transport plans. We can say something even stronger, namely, that for any sequence of plans $\pi_n$ weakly* converging to $\pi$, the liminf inequality for the costs holds, that is,
\begin{equation}\label{liminfinequality}
\pi_n \rightharpoonup \pi \qquad \Longrightarrow  \qquad \langle c_r,\pi\rangle \leq \liminf_{n\to \infty} \langle c_r,\pi_n\rangle\,.
\end{equation}\par\bigskip

Before concluding this introductory description, it is important to underline here two things. First of all, one is easily lead to guess that the correct rectification $c_r$ is simply the lower semi-continuous envelope of $c$. In fact, $c_r$ coincides with the l.s.c. envelope of $c$ in the two examples that we presented above, and moreover for a l.s.c. function the property~(\ref{liminfinequality}) is clearly always true. However, it is also easy to realize that the l.s.c. of $c$ \emph{does not} work as we want. To see this, it is enough to consider the situation when $X=Y=[0,1]$, $\mu=\nu$ is the Lebesgue measure, and
\[
c(x,y) = \left\{\begin{array}{ll} 0 &\hbox{if $(x, y)\in \Q\times\Q$}\,,\\ 1 &\hbox{otherwise}\,.\end{array}\right.
\]
In this case, the value of the cost function is almost surely $1$, so the problem is perfectly equivalent to the trivial problem with $c\equiv 1$, hence the duality already holds, the minimum is already attained, and there is no need to change anything. But on the other hand, the lower semi-continuous envelope of $c$ is costantly $0$. Looking at this problem, one easily understands what is wrong with the l.s.c. envelope. Roughly speaking, one needs to have $c_r(x,y)<c(x,y)$ if there are ``many'' points around $(x,y)$ with a low value of the cost function, while the lower semi-continuous envelope goes down even if there are only ``few'', but infinitely close, such points.\par

The second thing that we want to underline is, whether or not the rectification $c_r$ of $c$ depends on the measures $\mu$ and $\nu$. On one hand, it seems quite reasonable, and it would be of course much better, if it is not the case and $c_r$ depends only on $c$. But on the other hand, it is also easy to realize that this is not possible in general. In fact, if for instance $\mu$ and $\nu$ are concentrated on two points $\bar x\in X$ and $\bar y \in Y$, then the value of $c$ out of $(\bar x,\bar y)$ does not play any role and it cannot affect the definition of $c_r$. More precisely, we can observe that the fact whether there are ``many'' of ``few'' points around $(x,y)\in X\times Y$ of course depends on the measures $\mu$ and $\nu$. In fact, we will show (see Remark~\ref{dependenceonmunu}) that the rectification $c_r$ of $c$ only depends on the class of negligible sets with respect to $\mu$ and $\nu$, which is the best one could hope in view of the above considerations.

\subsection{Formal statement of our result}

In this section, we can give the formal definition of the rectification $c_r$ of $c$ and the correct statement of our main result. First of all, we need to introduce the following notion.
\begin{definition}\label{L-negligibility}
A set $A\subseteq X\times Y$ is called \emph{$L-$negligible} if there exist two sets $M\subseteq X$ and $N \subseteq Y$ with $\mu(M)=\nu(N)=0$ such that
\[
A\subseteq (M\times Y) \cup (X\times N) \,.
\]
Accordingly, if a property holds on the complement of an $L-$negligible set, than we sey that it holds \emph{$L-$almost surely}.
\end{definition}
It is trivial but fundamental to observe that the transport problem is not affected if the cost function is changed on an $L-$negligible set.\par

We can now give our definition of the rectified cost function.

\begin{definition}\label{RegDef}
Let $c:X\times Y\to [0,\infty]$ be measurable. A function $c_r:X\times Y\to [0,+\infty]$ is said to be the \emph{rectification of $c$} if the following holds:
\begin{enumerate}
\item[(i)] for all Borel functions $\phi,\psi:[0,1]\to [-\infty,\infty)$ satisfying $\phi\oplus \psi \leq c$ we have $\phi\oplus \psi \leq c_r$ $L-$almost surely;
\item[(ii)] $c_r$ is minimal subject to (i), i.e.\ if $d$ is another function satisfying (i) then $L-$almost surely $ c_r\leq d$.   
\end{enumerate}
\end{definition}
It is clear from (ii) that every cost function has at most one rectification, while the existence is not obvious. We can now state our result.

\begin{theorem}\label{MainTheorem} Take two Polish spaces $X$ and $Y$, two probability measures $\mu$ and $\nu$ on $X$ and $Y$ respectively, and a Borel measurable cost function $c:X\times Y\to [0,\infty]$. Then the following holds.
\begin{enumerate}
\item[({\bf A})] There exists a ($L-$almost surely) unique rectification $c_r$ of $c$. Moreover
\begin{enumerate}
\item[({\bf A1})] one has $L-$almost surely $c_r\leq c$;
\item[({\bf A2})] if $c$ is lower semi-continuous, then $L-$almost surely $c_r= c$;
\item[({\bf A3})] for the transport problem associated to $c_r$ duality holds, in particular
\[
P_{c_r} = D_{c_r} = D_c\,.
\]
\end{enumerate}
\item[({\bf B})] The transport problem associated to $c_r$ admits a solution (i.e., an optimal transport plan). 
Moreover
\begin{enumerate}
\item[({\bf B1})] for any transport plan $\pi$ and for any sequence of transport plans $\pi_n \rightharpoonup \pi$ one has
\[
\iint_{X\times Y}   c_r \, d\pi \leq \liminf_{n\to\infty} \iint_{X\times Y}   c_r \, d\pi_n\,;
\]
\item[({\bf B2})] for any transport plan $\pi$ there is a suitable sequence of measures $\pi_n \rightharpoonup \pi$ so that 
\[
\iint_{X\times Y}   c_r \, d\pi = \lim_{n\to\infty} \iint_{X\times Y}   c \, d\pi_n\,.
\]
\end{enumerate}
\item[({\bf C})] There exist Polish topologies $\tau_X, \tau_Y$ on $X$ resp.\ $Y$ which refine the original topologies, lead to the same Borel sets and are  so that $ c_r$ is lower semi-continuous w.r.t.\ $\tau_X\otimes \tau_Y$. 
\end{enumerate}
\end{theorem} 

\begin{remark}\label{BLSDuality}
We underline that another way of ``solving'' the situations where the duality does not hold has been given in~\cite{BeLS09a}.
For $\eps>0$, consider the $1-\eps$ partial transportation problem
\[
P^\eps_c:= \inf \left\{\iint c\,d\pi:P_X\pi\leq \mu, P_Y\leq \nu, \|\pi\|\geq 1-\eps\right\}\,.
\]
Then~\cite[Theorem 1.2]{BeLS09a} asserts that
\[
D_c= P^{\textrm{relaxed}}_c:= \lim_{\eps\downarrow 0}P_c^\eps\,.
\]
\end{remark}

\section{Proof of the main result}

In this section, we prove the main Theorem and we add some remarks and examples. We start with one of the main ingredients of the proof, namely, to show the existence of a rectification $c_r$ corresponding to the cost function $c$. In fact, we can show something more precise.

\begin{lemma}\label{MainLemma}
There exists a unique rectification $c_r: X\times Y\to [0,\infty]$ of $c$. Moreover, there exist two sequences of measurable and bounded functions $\phi_n:X\to \R$ and $\psi_n:Y\to \R$ such that $\phi_n\oplus \psi_n \leq c$ for all $n$, and
\[
c_r=\sup_{n\geq 1} \phi_n\oplus \psi_n\,.
\]
\end{lemma} 
In the proof of this result, we will use the following characterization of $L$-negligible sets.
\begin{lemma}\label{KellLemma}
A Borel set $A\subseteq X\times Y$ is $L-$negligible if and only $\pi(A)=0$ for every transport plan $\pi\in\Pi(\mu, \nu)$. 
\end{lemma}
\begin{proof} If $A$ is $L-$negligible, then clearly  $\pi(A)=0$ for every transport plan $\pi\in \Pi(\mu, \nu)$. The other direction is more difficult and was  first established by Kellerer as a  consequence of the Duality Theorem for bounded cost functions \cite[Proposition 3.5]{Kell84}. See also \cite[Appendix A]{BeLS09a} for a more direct proof.
\end{proof}

\begin{proof}[Proof of Lemma \ref{MainLemma}.]
As already noticed, the uniqueness of a rectification is trivial by property~(ii) of Definition~\ref{RegDef}, hence we have only to show the existence. For simplicity, we will divide the proof of the lemma in some steps.
\step{I}{Reduction to the case of a bounded cost function $c$.}
We start the proof by reducing to the case of a bounded function $c$. First of all, for any function $\tau$ and any $n\in \N$ let us set
\[
\tau^{(n)}:=\max\Big(\min(\tau,n),-n\Big)\,.
\]
It is now immediate to notice that, for any two functions $\phi:X\to[-\infty,\infty), \psi:Y\to[-\infty,\infty)$, the sequence
\[
n\mapsto \phi^{(n)}\oplus \psi^{(n)}
\]
is increasing where $\phi\oplus \psi$ is positive, and it is pointwise converging to $\phi\oplus\psi$. Hence, a function $d:X\times Y\to [0,\infty]$ satisfies $L-$almost surely the inequality $\phi\oplus \psi\leq d$ if and only if $\phi^{(n)}\oplus\psi^{(n)}\leq d$ for all $n\in \N$. As a consequence, if a function $d$ satisfies property~(i) of Definition~\ref{RegDef} for all the pairs of functions which are bounded, then it already satisfies~(i) in full generality. We are then ready to show the claim of this step. Indeed, let us assume that the lemma has been already established for all bounded cost functions, and pick a generic cost function $c$. By assumption, for any $n$ we know that $c^{(n)}$ admits a rectification $c^{(n)}_r= \sup_{j\in \N} \phi_{n,j}\oplus \psi_{n,j}$. We claim then that
\[
c_r := \sup_{n,j} \phi_{n,j}\oplus \psi_{n,j} = \sup_n c^{(n)}_r
\]
is a rectification of $c$. Concerning property~(i), for all Borel functions $\phi,\,\psi$ we have
\[\begin{split}
\phi\oplus \psi \leq c &\quad\Longrightarrow\quad \phi^{(n)}\oplus \psi^{(n)} \leq c\quad \forall\, n
\quad\Longrightarrow\quad \phi^{(n)}\oplus \psi^{(n)} \leq c^{(2n)}\quad \forall\, n\\
&\quad\Longrightarrow\quad \phi^{(n)}\oplus \psi^{(n)} \leq c^{(2n)}_r\leq c_r \quad \forall\, n
\quad\Longrightarrow\quad \phi\oplus \psi \leq c_r\,.
\end{split}\]
On the other hand, concerning property~(ii), let $d:X\times Y\to [0,+\infty]$ satisfy~(i), and let $n\in \N$. Since we assume the validity of the lemma for $c^{(n)}$, which is bounded, from the fact that
\[
\phi\oplus \psi \leq c^{(n)} \quad\Longrightarrow\quad \phi\oplus \psi \leq c
\quad\Longrightarrow\quad \phi\oplus \psi \leq d\,,
\]
we immediately deduce that $c^{(n)}_r \leq d$. Hence, clearly $c_r = \sup_n c^{(n)}_r \leq d$.

\step{II}{The bounded case: definition of $c_r$ and property~(ii).}
In view of Step~I, let us now concentrate on the case of a bounded cost function $c$, say $c:X\times Y\to [0,M]$.  Consider the set
\[
V:=\bigg\{(f,g): \,f:X\to [0,1],\, g:Y\to [0,1], \int f\, d\mu=\int g\, d\nu\bigg\}\,, 
\]
and pick a family $\{(f_n,g_n)\}_{n\in\N}\subseteq V$ which is dense in $V$  in the sense that for all $(f,g) \in V$ and  $\eps>0$ there are $f_n, g_n$ satisfying $\|f-f_n\|_1+\|g-g_n\|_1\leq \eps$.   
 For each $n\in\N$, let us take a pair of functions $(\phi_n,\psi_n)$ which are optimal for the dual problem, hence such that $\phi_n\oplus \psi_n\leq c$ and
\begin{equation}\label{NeedBD}
\int \phi_n\, d(f_n\mu)+\int \psi_n\,d(g_n\nu)=\inf\left\{ \iint c\, d\gamma: P_X\gamma=f_n\nu,P_Y\gamma=g_n\mu\right\}.
\end{equation}
This is possible thanks to the known duality theorem for bounded cost functions (\cite[Theorem 2.21]{Kell84}).

For technical reasons it will be  convenient to take also the pair of functions $\phi_0\equiv 0, \psi_0\equiv 0$ into account.

Let us now define $c_r:= \sup_{n\geq 0} \phi_n\oplus\psi_n$: The proof will be obtained by checking that $c_r$ is a rectification of $c$. Let us start with the minimality property~(ii), which is straightforward. Indeed, let $d$ be a function which satisfies~(i). For any $n\geq 0$, then, by construction we have $\phi_n\oplus \psi_n\leq c$, and by~(i) this implies $\phi_n\oplus \psi_n\leq d$. Passing to the supremum, we obtain $c_r = \sup_n \phi_n\oplus \psi_n \leq d$, then the required minimality property~(ii).
\step{III}{The bounded case: proof of~(i).}
In view of the preceding steps, we still only have to check that the function $c_r$ defined above verifies~(i). Striving for a contradiction, we assume that there exist functions $\phi:X\to\R, \psi:Y\to \R, \phi\oplus \psi \leq c$ such that the set $\{\phi\oplus \psi > c_r\}$ is not $L-$negligible. 

Pick, by Lemma \ref{KellLemma}, a transport plan $\pi_0\in \Pi(\mu,\nu)$ so that $\pi_0(\{\phi\oplus \psi > c_r\})>0$.  As 
\[
\{\phi\oplus \psi > c_r\}=\bigcup_{a,b,\delta \in \Q, \delta > 0}
 \{(x,y):\phi(x)>a,\psi(y)>b, a+b>c_r(x,y) + \delta \}\,,
\]
there exist $a,b\in\R,\delta >0$ and a Borel set $\Gamma\subseteq X\times Y$ so that 
\begin{align*}
\pi_0(\Gamma)&>0  \\
a&<  \phi &\mbox{ on } A:=P_X \Gamma,\\  
b&<  \psi &\mbox{ on } B:=P_Y \Gamma,\\
c_r&< a+b-\delta & \mbox{ on }\Gamma. 
\end{align*}
Let now
\begin{align*}
\gamma_0:=\pi_0{\upharpoonright \Gamma}\,, &&
f:=\frac{d(P_X\gamma_0)}{d\mu}\,, &&
g:=\frac{(dP_Y\gamma_0)}{d\nu}\,,
\end{align*}
where the first definition means that for any Borel set $\Delta$ one has
\[
\gamma_0(\Delta) = \pi_0\big( \Gamma\cap \Delta)\,.
\]
Since $(f,g)\in V$, we can pick  $n\geq 1$ so that  $f_n,\, g_n$ satisfy
\begin{equation}\label{DCBound}
\|f-f_n\|_1+\|g-g_n\|_1  < \frac{ \delta \|\gamma_0\|}{2M}\,.
\end{equation}
Recalling now that $c\geq \phi\oplus \psi> a+b$ on $A\times B$, we can estimate
\begin{equation}\label{FirstDC}\begin{split}
\inf_{\gamma\in\Pi(f_n\mu,g_n\nu)} \iint_{X\times Y} c\, d\gamma
&\geq \inf_{\gamma\in\Pi(f_n\mu,g_n\nu)} \iint_{A\times B} c\, d\gamma\\
&\geq (a+b)\bigg(\|\gamma_0\| -\|f-f_n\|_1-\|g-g_n\|_1\bigg)\,.
\end{split}\end{equation}
On the other hand, set
\begin{align*}
\alpha := \frac{d(f_n\mu)}{d(f\mu)}\wedge 1\,, && \beta:= \frac{d(g_n\nu)}{d(g\nu)}\wedge 1\,, &&
\tilde\gamma_0 = \big( \alpha\wedge \beta ) \gamma_0 \leq \gamma_0\,,
\end{align*}
and notice that
\begin{equation}\label{notice1}\begin{split}
\|\gamma_0\| - \| \tilde\gamma_0\| &= \iint 1- \big(\alpha\wedge \beta\big)\, d\gamma_0
\leq \iint 1-\alpha\, d\gamma_0 + \iint 1-\beta\, d\gamma_0\\
&\leq \|f-f_n\|_1+\|g-g_n\|_1\,.
\end{split}\end{equation}
We can then call
\begin{align*}
\tilde f := \frac{d(P_X \tilde\gamma_0)}{d\mu}\,, &&
\tilde g := \frac{d(P_Y \tilde\gamma_0)}{d\nu}\,, &&
f_r := f_n - \tilde f \geq 0\,, &&
g_r := g_n - \tilde g \geq 0\,,
\end{align*}
observe that
\begin{equation}\label{notice2}
\| f_r\mu \| = \| g_r\nu \| \leq \|f-f_n\|_1+\|g-g_n\|_1\,,
\end{equation}
thus getting to evaluate
\[\begin{split}
\int \phi_n \,&d (f_n\mu) + \int \psi_n\,d (g_n \nu)\\
&=\int \phi_n \,d (\tilde f\mu) + \int \psi_n\,d (\tilde g \nu)+\int \phi_n \,d (f_r\mu) + \int \psi_n\,d (g_r \nu)\\
&=\iint \phi_n\oplus\psi_n \,d \tilde\gamma_0+\iint \phi_n\oplus\psi_n \,d \bigg(\frac{(f_r\mu)\otimes (g_r \nu)}{\|f_r\mu \|}\bigg)\\
&\leq \iint c_r \,d \tilde\gamma_0+ M \|f_r\mu\|
\leq \iint c_r \,d\gamma_0 + M \|f_r\mu\|   \\
&\leq \big(a+b-\delta \big) \| \gamma_0 \|+ M \Big(\|f-f_n\|_1+\|g-g_n\|_1\Big)\,,
\end{split}\]
where we have used~(\ref{notice1}), (\ref{notice2}) and the fact that $c_r\geq 0$, which immediately comes from the definition of  $\phi_0$ and $\psi_0$. 
Finally, inserting the last inequality and~(\ref{FirstDC}) into~(\ref{NeedBD}), and recalling that by construction $a+b\leq \sup c \leq M$, we readily obtain
\[
\|f-f_n\|_1+\|g-g_n\|_1 \geq \frac{\delta \| \gamma_0 \|}{2M} \,,
\]
which together with~(\ref{DCBound}) provides the searched contradiction.
\end{proof}

Before coming to the proof of Theorem~\ref{MainTheorem}, we can underline what follows.

\begin{remark}\label{dependenceonmunu}
It is important to notice that the dependence of the rectification $c_r$ on $\mu$ and $\nu$ is in fact only a dependence on the class of the negligible sets on $X$ and $Y$ with respect to $\mu$ and $\nu$. This is obvious from Definitions~\ref{RegDef} and~\ref{L-negligibility}, since everything depends on which sets are $L-$negligible, and in turn this only depends on the $\mu-$ and $\nu-$negligible sets.
\end{remark}

\begin{proof}[Proof of Theorem \ref{MainTheorem}.]
Let us start from Property~({\bf C}). 
Our argument will be based on  \cite[Theorem 13.1]{Kech95}: if $Z$ is a polish space and $B_1, B_2, \ldots$ are Borel sets in $Z$, then there exists a Polish topology $\tau$ on $Z$ so that $\tau$ refines the original topology, $\tau$ generates the same Borel sets as the original topology and all sets $ B_n, n\in\N$ are open in $\tau$.

A useful application is that a Borel function can be viewed as continuous function on a modified space. More precisely, if $f:Z\to\R$ is Borel, then we can apply the above-mentioned result to the sets $B_n:=f^{-1} (U_n), n\in\N$, being $\{U_n\}_{n\in\N}$ a neighborhood basis of $\R$, to obtain a Polish topology $\tau$ on $Z$ which refines the original topological and in which all sets $B_n$ are open. Consequently, $ f$ is a  continuous function on the space $(Z,\tau)$.

By Lemma \ref{MainLemma}, there exist two sequences of measurable functions $\phi_n:X\to \R$ and $\psi_n:Y\to\R$  such that $\sup_{n\geq 1} \phi_n\oplus \psi_n =c_r$. 
Using the just explained argument,  we can find topologies $\tau_X$ and $\tau_Y$ on $X$ resp.\ $Y$ so that, for every $n\in\N$, $\phi_n$ and $\psi_n$  are  continuous functions on $(X,\tau_X)$ resp.\ $(Y,\tau_Y)$. As a consequence, the functions $\phi_n\oplus\psi_n, n\in\N$  are continuous  on $(X\times Y, \tau_X\otimes \tau_Y)$ and hence $c_r=\sup_{n\geq 1} \phi_n\oplus \psi_n$ is l.s.c.\ with respect to $\tau_X\otimes  \tau_Y$, so that Property~({\bf C}) follows.

\medskip

Let us now consider Property~({\bf A}). The existence and uniqueness of a rectification have been already established with Lemma~\ref{MainLemma}. Concerning Property~({\bf A1}), it is clear from the definition of the rectification.\par
Let us then consider Property~({\bf A2}). Pick families of open sets  $\{U_n\}_{n\in\N}$ and $ \{V_m\}_{m\in\N}$ which form  bases of the topologies of $X$ resp.\ $Y$. For $n,m\in\N$, set 
\[
e_{n,m}:= \inf_{(x,y)\in U_n\times V_m} c(x,y)
\] 
and define $\phi_{n,m}:X\to \R$ and $\psi_{n,m}:Y\to \R$ so that
\begin{align*}
\phi_{n,m}\oplus \psi_{n,m} 
&=  e_{n,m} &\mbox{ on } U_n\times V_m\\
\phi_{n,m}\oplus \psi_{n,m}& \leq  0&  \mbox{ otherwise. }
\end{align*}
Since by construction $\phi_{n,m}\oplus \psi_{n,m}\leq c$, by definition we have $\phi_{n,m}\oplus \psi_{n,m}\leq c_r$, and hence also
\[
c_r \geq \sup_{n,m} \phi_{n,m}\oplus \psi_{n,m}\,.
\]
Finally, if $c$ is l.s.c. then the latter supremum coincides with $c$ itself, so $L-$a.s. one has $c_r\geq c$, which together with Property~({\bf A1}) concludes the searched equality.\par
Finally, we consider Property~({\bf A3}). First of all, we can observe that $c$ and $c_r$ have the same dual problem, that is, $D_c=D_{c_r}$. To do so, take two functions $\varphi,\,\psi$, integrable with respect to $\mu$ and $\nu$ respectively, and such that $\phi\oplus \psi\leq c$. By definition, there exist sets $M\subseteq X, N \subseteq Y, \mu(M)=\nu(N)=0$ so that $(\phi\oplus\psi)(x,y)\leq c_r(x,y)$ for all $x\in X\setminus M, y\in Y\setminus N$, hence for $\tilde \phi:=\phi-I_M, \tilde \psi:=\psi-I_N$ we have $\tilde \phi\oplus\tilde\psi\leq c_r$ and $\int \phi\,d\mu=\int\tilde\phi\,d\mu, \int\psi\,d\nu=\int \tilde\psi\,d\nu$. This shows $D_{c_r}\geq D_c$. The other inequality is identical. Indeed, if $\varphi,\,\psi$ are integrable and $\phi\oplus \psi\leq c_r$, by Property~({\bf A1}) there are again two sets $M\subseteq X, N \subseteq Y, \mu(M)=\nu(N)=0$ so that $(\phi\oplus\psi)(x,y)\leq c(x,y)$ for all $x\in X\setminus M, y\in Y\setminus N$, so exactly as before we get $D_c \geq D_{c_r}$, and in particular we have $D_c=D_{c_r}$.\par
Moreover, having already established Property~({\bf C}), the equality $D_{c_r}=P_{c_r}$ comes directly from the standard duality theorem for l.s.c. cost functions (notice that a change of the topology which does not change the Borel sets does not effect neither the primal nor the dual problem).

\medskip 

We are then finally left with Property~({\bf B}). First of all, the existence of an optimal transport plan with respect to $c_r$ is obvious by Property~({\bf C}). Indeed, it is well-known that a transport problem with a l.s.c. cost admits an optimal transport plan, and the optimality of a plan is again not effected, of course, by a change of the topology.\par
Let us now consider Property~({\bf B1}). It is very well-known that, whenever $d$ is a l.s.c. function, and $\gamma_n$ is a sequence of measures weakly* converging to $\gamma$, one has
\[
\int d\, d\gamma \leq \liminf_{n\to\infty} \int d\, d\gamma_n\,.
\]
Therefore, since $c_r$ is l.s.c. with respect to $\tau_X\otimes \tau_Y$, we immediately get
\[
\iint_{X\times Y}   c_r \, d\pi \leq \liminf_{n\to\infty} \iint_{X\times Y}   c_r \, d\pi_n
\]
for any sequence $\{\pi_n\}$ which weakly* converges to $\pi$ with respect to the new topologies. We will conclude as soon as we observe that this latter weak* convergence is equivalent to the original one. We will obtain this equivalence thanks to Lemma~\ref{newmathias} below. Indeed, it is immediate to observe that, since the topologies $\tau_X$ and $\tau_Y$ are finer than the original topologies on $X$ and $Y$, then also the weak* topology on $\Pi(\mu,\nu)$ is finer then the original one, so to apply Lemma~\ref{newmathias} we only need the weak* compactness (in the ``new'' sense) of the set $\Pi(\mu,\nu)$ of the transport plans. And in turn, this compactness can be derived from Prohorov's Theorem as in \cite[p55-57]{Vill09}. 
\par
Finally, we are left with Property~({\bf B2}). Keeping in mind Property~({\bf B1}), it is sufficient to show that for every $\pi\in\Pi(\mu,\nu)$ there exists a sequence $\pi_n\rightharpoonup \pi$ so that
\begin{equation}\label{AimBelow}
\limsup_{n\to\infty} \iint c\, d\pi_n\leq \iint c_r\, d\pi\,.
\end{equation}
We present the proof for the case $X=Y= [0,1]$ and $\mu=\nu=\lambda$, because then the argument is much simpler to read, but at the end it will be clear that the proof of the general case is equivalent, and just more notationally unconfortable.\par
If $\iint_{X\times Y} c_r\,d\pi=\infty$ there is nothing to prove, so assume that $\iint_{X\times Y} c_r\,d\pi<\infty$.
Fix $n\in\N$ and $l,m\in\{1,\ldots,n\}$.
Set 
\[
D_{l,m}^n:=\left(\tfrac{l-1}n,\tfrac ln\right]\times \left(\tfrac{m-1}n,\tfrac mn\right].
\]
Denote by $\mu_{l,m}^n , \nu_{l,m}^n$ the marginals of $\pi\upharpoonright D_{l,m}^n$. 
For $\mu_{l,m}^n-$, resp.\ $ \nu_{l,m}^n-$integrable functions $\phi: \left(\tfrac{l-1}n,\tfrac ln\right]\to \R,\psi:\left(\tfrac{m-1}n,\tfrac mn\right]\to\R$ satisfying $\phi\oplus \psi\leq c$ we have 
\[
\int \phi \, d\mu_{l,m}^n+ \int \psi \, d\nu_{l,m}^n \leq  \iint_{D_{l,m}^n} c_r \,d\pi< \infty\,.
\]
Hence the optimal dual value corresponding to the cost function $c$ and the spaces  $\left(\left(\tfrac{l-1}n,\tfrac ln\right],\mu_{l,m}^n\right)$,  $\left(\left(\tfrac{m-1}n,\tfrac mn\right],\nu_{l,m}^n\right)$ is finite.  
By Remark~\ref{BLSDuality} (resp.\  \cite[Theorem 1.2]{BeLS09a}), there exist $\mu_{l,m}^n-$, resp.\ $ \nu_{l,m}^n-$integrable functions $\phi_{l,m}^n: \left(\tfrac{l-1}n,\tfrac ln\right]\to \R,\psi_{l,m}^n:\left(\tfrac{m-1}n,\tfrac mn\right]\to\R$ and a measure $\pi_{l,m}^n$ on $D_{l,m}^n$ so that
\begin{align}
P_X\pi_{l,m}^n\leq \mu_{l,m}^n\,,\quad P_Y\pi_{l,m}^n\leq \nu_{l,m}^n\,, \quad\|\pi_{l,m}^n\|& \geq \|\mu_{l,m}^n\| -\tfrac1{n^3}\,, \label{SquareApprox} \\
\label{LittleSquareAdmissible}\phi_{l,m}^n\oplus\psi_{l,m}^n& \leq c\,,\\ 
\label{LittleSquareDuality} 
\int \phi_{l,m}^n\, d\mu_{l,m}^n + \int \psi_{l,m}^n\, d\nu_{l,m}^n&\geq \iint c \,d\pi_{l,m}^n-\tfrac1{n^3}\,. 
\end{align}
Define a measure $\pi_n$ on $X\times Y$ by the requiring that $\pi_n\upharpoonright D_{l,m}^n=\pi_{l,m}^n$ for all $l,m\in \{1,\ldots,n\}$. 
It follows from  \eqref{SquareApprox} that
\[
\lim_{n\to\infty} \pi_n(D^{n}_{l,m})=\pi(D^n_{l,m})
\]
for all $l,m,n\in\N, l,m\leq n$. Consequently, $(\pi_n)_{n\geq 1}$ converges weakly* to $\pi$. 
From  \eqref{LittleSquareAdmissible} and \eqref{LittleSquareDuality} we deduce 
\begin{align*}
\begin{split} \label{AlmostAimBelow}
\iint c\,d \pi_n-\tfrac1{n}
 &  = \sum_{l,m\leq n}\Big( \iint c \,d\pi_{l,m}^n-\tfrac1{n^3}\Big)\\
 &  \leq \sum_{l,m\leq n}  \int \phi_{l,m}^n\, d\mu_{l,m}^n + \int \psi_{l,m}^n\, d\nu_{l,m}^n\\
& = \sum_{l,m\leq n}  \iint_{D_{l,m}^n} \phi_{l,m}^n\oplus  \psi_{l,m}^n\, d\pi\\
&\leq  \sum_{l,m\leq n}  \iint_{D_{l,m}^n} c_r\, d\pi = \iint c_r \,d\pi\,.
\end{split}
\end{align*}
Letting $n$ tend to $\infty$ in the last inequality, we obtain  \eqref{AimBelow}. 
\end{proof}

In the above proof, we needed to use the following technical topological result.

\begin{lemma}\label{newmathias}
Let $(Z,\tau)$ be a (Hausdorff-) topological space, and let $\tilde\tau$ be a finer topology. Then $\tau$ and $\tilde\tau$ agree on each subset of $Z$ which is compact with respect to $\tilde\tau$.
\end{lemma}
\begin{proof}
Let $K$ be a $\tilde\tau-$compact subset of $Z$, and let $T:(K,\tilde\tau)\to (K,\tau)$ be the identity map. Since $\tilde\tau$ is finer than $\tau$, then $T$ is continuous. On the other hand, let $A\subseteq K$ be $\tilde\tau-$open. Hence, $C=K\setminus A$ is $\tilde\tau-$closed, hence $\tilde\tau-$compact because so is $K$. Being $T$ continuous, and since the continuous images of compact sets are compact, we have that $C$ is also $\tau-$compact, hence $\tau-$closed. In conclusion, the generic $\tilde\tau-$open set $A$ is also $\tau-$open, and this shows that the two topologies agree on $K$.
\end{proof}

\begin{remark}
It is interesting to observe that the rectification $c_r$ of $c$ is almost characterized by Property~({\bf B}) of Theorem~\ref{MainTheorem}.

\medskip

More precisely, it is apparent  from the proof of Property~({\bf B})  that $c_r $ satisfies the following slightly stronger assertion:
\begin{enumerate}
\item[({\bf B'})]
for any measure $\pi$ on $X\times Y$, $P_X\pi\leq \mu,P_Y\pi\leq \nu$ and any sequence $\pi_n \rightharpoonup \pi$, $P_X\pi_n\leq \mu,P_Y\pi_n\leq \nu$ one has
\[
\int_{X\times Y}   c_r \, d\pi \leq \liminf_{n\to\infty} \int_{X\times Y}   c_r \, d\pi_n\leq \liminf_{n\to\infty} \int_{X\times Y}   c \, d\pi_n\,,
\]
and moreover for any such measure $\pi$ there is a sequence $\pi_n \rightharpoonup \pi$ such that the above inequalities are equalities.
\end{enumerate}

Assume now that $\tilde c_r $ is another function satisfying the ({\bf B'}), then $c_r=\tilde c_r $  $L-$almost surely.

To see this observe that ({\bf B'}) implies 
$$\iint c_r\, d\pi=\iint\tilde c_r\, d\pi$$ for every $\pi$ satisfying $P_X\pi_n\leq \mu,P_Y\pi_n\leq \nu$. In turn also $\pi(\{c_r<  \tilde c_r\})=\pi(\{c_r> \tilde c_r\})=0$ for every $\pi\in \Pi(\mu,\nu)$. Our claim then follows from Lemma~\ref{KellLemma}.
\end{remark}

\begin{remark}
In light of Theorem~\ref{MainTheorem}, it is natural to ask whether there is in general a modification of $c_r$ on an $L-$negligible set that is  l.s.c.\ already in the original topology, resp.\ whether there is a nice way to define a rectification-procedure which has this feature. The following simple example shows why this is not the case in general. 

Let $(X,\mu)=(Y,\nu)=([0,1],\lambda)$. Let $(q_n)_{n\geq 1}$ be an enumeration of the rationals in $[0,1]$. Pick $\alpha$ so that 
$$D:= [0,1]\setminus \bigcup_{n\geq 1} (q_n-\alpha/2^n,q_n+\alpha/2^n)$$
has Lebesgue-measure $1/2$.  Set $\phi=I_D,\psi\equiv 0$ and $c\equiv \phi\oplus\psi.$ Then $P_c=D_c=1/2$, but every lower semi-continuous function $g:X\times Y\to [0,\infty]$ which is  $L-$almost surely smaller than $c$ necessarily satisfies $g\leq 0$ $L-$almost surely.
\end{remark}
\def\ocirc#1{\ifmmode\setbox0=\hbox{$#1$}\dimen0=\ht0 \advance\dimen0
  by1pt\rlap{\hbox to\wd0{\hss\raise\dimen0
  \hbox{\hskip.2em$\scriptscriptstyle\circ$}\hss}}#1\else {\accent"17 #1}\fi}

\end{document}